\RequirePackage{fix-cm}
\documentclass[smallcondensed]{svjour3}     % onecolumn (ditto)
\smartqed  % flush right qed marks, e.g. at end of proof
\usepackage{graphicx}
\usepackage{enumerate}
\usepackage{amsmath,amssymb}
\usepackage{tikz, hyperref}
\usetikzlibrary{arrows.meta}

\newtheorem{lem}{Lemma}
\newtheorem{thm}[lem]{Theorem}

\begin{document}

\title{A Classification of Hyperfocused 12-Arcs
\thanks{Research supported in part by U.S. National Science Foundation grant DMS-1427526 for The Rocky Mountain--Great Plains Graduate Research Workshop in Combinatorics, and Collaboration Grants from the Simons Foundation (\#316262 to Stephen G. Hartke, \#711898 to Jason Williford).}

}


\author{Philip DeOrsey \and Stephen G. Hartke \and Jason Williford}


\institute{P. DeOrsey \at
                Department of Mathematics\\
              Westfield State University \\
              577 Western Ave\\ 
              Westfield, MA 01086 USA\\
              \email{pdeorsey@westfield.ma.edu}           %  \\
           \and
           S.G. Hartke \at
              Department of Mathematical and Statistical Sciences\\
              University of Colorado Denver\\
              Campus Box 170\\
              P.O. Box 173364\\
              Denver, Colorado 80217-3364 USA\\
              \email{stephen.hartke@ucdenver.edu}
              \and
              J. Williford \at
              Deparment of Mathematics\\
              University of Wyoming\\
              Dept 3036\\
              1000 E University Ave\\
              Laramie, WY 82071
}

\date{Received: date / Accepted: date}

\maketitle

\begin{abstract}
A $k$-arc in PG($2,q$) is a set of $k$ points no three of which are collinear. A hyperfocused $k$-arc is a $k$-arc in which the $k \choose 2$ secants meet some external line in exactly $k-1$ points. Hyperfocused $k$-arcs can be viewed as 1-factorizations of the complete graph $K_k$ that embed in PG($2,q$). We study the 526,915,620 1-factorizations of $K_{12}$, determine which are embeddable in PG($2,q$), and classify hyperfocused $12$-arcs. Specifically we show if a $12$-arc $\mathcal{K}$ is a hyperfocused arc in PG($2,q$) then $q = 2^{5k}$ and $\mathcal{K}$ is a subset of a hyperconic including the nucleus. 
%\subjclass{Primary 51E20; Secondary 05B25}
\end{abstract}

\section{Introduction}

A {\it secret sharing scheme} is a method of assigning information to participants so that only certain subsets of the participants can combine their information to reveal a secret. In a specific type of $2$-level secret sharing scheme there are two tiers of participants; any two participants in the top tier can combine their information to reveal the secret, but revealing the secret requires three participants from the bottom tier. Additionally, two participants from the bottom tier can be joined by one from the top tier to determine the secret. It is well known that we can use projective geometry to develop such a scheme \cite{Simmons}.  The following is a description of a scheme that uses the incidence of specific planes in PG($4,q$).

Consider two tangent planes $\pi_1$ and $\pi_2$ in PG($4,q$) and allow the secret to be the point $P = \pi_1 \cap \pi_2$. Make it known to all participants that $P$ is in $\pi_1$ while $\pi_2$ remains a secret. Choose a line $\ell$ in $\pi_2$ passing through $P$ and assign points from $\ell$ to participants in tier one so that any two can combine their points, span $\ell$, and determine the secret $P = \ell \cap \pi_1$. For the participants in tier two, choose an arc $\mathcal{K}$ in $\pi_2$ disjoint from $\ell$ and assign points on $\mathcal{K}$ to participants so that any three can combine their points, span $\pi_2$ and determine the secret $P = \pi_1 \cap \pi_2$. In order to allow any two participants from tier two to combine with any one from tier one we must ensure that the points assigned from $\ell$ are not on any secant lines to $\mathcal{K}$, so that the three points span $\pi_2$. Hence, it is to our benefit to pick $\mathcal{K}$ so that the lines secant to $\mathcal{K}$ meet $\ell$ in a minimum number of points.

In order to maximize the assignable pieces of information, Simmons \cite{Simmons} investigated $k$-arcs with the property that all of their secants met some external line in exactly $k$ points. He called these arcs {\it sharply focused sets}, and studied their existence in certain projective planes of small order. In \cite{Holder} Holder continued this investigation and considered the case where the secants of a $k$-arc meet some external line in a set of $k-1$ points, the theoretical minimum. Holder called this type of arc a {\it super sharply focused set}, which was later changed to {\it hyperfocused arc} by Cherowitzo and Holder in \cite{Cherowitzo}.

It is clear that $k$ must be even for a hyperfocused $k$-arc to exist because secant lines must partition the arc. Since a 2-arc is trivially hyperfocused on any line, we say a hyperfocused $k$-arc is non-trivial if $k \geq 4$. Following from a result of Bichara and Korchm{\'a}ros \cite{Korchmaros}, Cherowitzo and Holder showed that if a non-trivial hyperfocused arc exists in PG($2,q$) then $q$ is even. Further, it was shown that if a hyperfocused $k$-arc is not a hyperoval or a hyperoval minus two points, then $k \leq \frac{q}{2}$.

Hyperfocused $k$-arcs were classified for $k = 4,6,8$ by Drake and Keating in \cite{Drake}, where they use the setting of Desarguesian nets, and independently by Cherowitzo and Holder in \cite{Cherowitzo} where they also classified hyperfocused 10-arcs. Cherowitzo and Holder exploited the relationship between the 1-factorizations of $K_k$ and hyperfocused $k$-arcs in order to complete their classifications. Hyperfocused 12-arcs were classified in PG($2,32$) by Faina et al. in \cite{Faina}, where they also showed the non-existence of hyperfocused 14-arcs in PG(2,32). 

We use the perspective of Cherowitzo and Holder to study hyperfocused 12-arcs by examining the 1-factorizations of $K_{12}$.  We show there is a unique 1-factorization of $K_{12}$ which embeds in PG($2,q$), $q = 2^h$, and so there is a unique hyperfocused $12$-arc, which is a subset of a hyperconic. Our method uses a computer search of the non-isomorphic 1-factorizations of $K_{12}$ along with a new necessary condition for a 1-factorization to embed as a hyperfocused arc.

\section{Definitions and Preliminaries}

A $k$-arc in PG($2,q$) is a set of $k$ points no three of which are colinear. It is well known that $k \leq q+1$ if $q$ is
odd and $k \leq q+2$ if $q$ is even. A ($q+1$)-arc is called an {\it oval}, and a ($q+2$)-arc is called a {\it hyperoval}. Based on the result of Cherowitzo and Holder we will restrict our attention to the case when $q$ is even, that is, $q = 2^h$. Recall that a conic is a set of points satisfying an irreducible homogeneous quadratic equation. The {\it nucleus} of an oval is the point where all of its tangent lines intersect. A {\it hyperconic} is a hyperoval consisting of a conic together with its nucleus.

Let $\ell$ be a line external to a $k$-arc $\mathcal{K}$. We say $\mathcal{K}$ is {\it hyperfocused on $\ell$} if the $k \choose 2$ secants of $\mathcal{K}$ meet $\ell$ in exactly $k-1$ points. The line $\ell$ is called the {\it focus line}, and the set of $k-1$ points in which the secants meet $\ell$ is called the {\it focus set}. If we identify the $k$ points of a hyperfocused $k$-arc with the vertices of a complete graph then the secant lines to the arc naturally correspond to the edges of the graph. For each point in the focus set, the secant lines through that point partition $\mathcal{K}$ into pairs. This corresponds to a set of disjoint edges that cover the vertices of the complete graph, which is called a {\it 1-factor}.

Thus, the focus set determines a {\it 1-factorization}, a set of disjoint 1-factors that cover the edges of the graph. Given a 1-factorization $\mathcal{F}$ of a complete graph, we will call that 1-factorization {\it embeddable} if there is a hyperfocused arc in some finite classical plane that determines $\mathcal{F}$. 

A 1-factorization of a complete graph induces an edge coloring of that graph, where each color class is a 1-factor. For convenience we will at times refer to this induced edge coloring to state facts about the 1-factorization. We say that two subgraphs have the same coloring if there is an isomorphism between them that preserves the colors of edges.

In order to determine which 1-factorizations of $K_{12}$ are embeddable in PG($2,q$) we will first develop some necessary conditions that will help us eliminate 1-factorizations that cannot be embedded. To discuss them properly we recall the concept of the diagonal line of a quadrangle ABCD. In PG($2,2^h$) the quadrangle ABCD can be completed to a Fano subplane. The line meeting this subplane external to the quadrangle is known as the {\it diagonal line}. Our first necessary condition comes from Cherowitzo and Holder. We use $C_4$ to denote the cycle on four vertices and $K_4$ to denote the complete graph on four vertices; see \cite{West} for any other graph theory terminology.

\begin{figure}
\centering
\begin{tikzpicture}[scale = 2]

\draw[color = blue, dashed, ultra thick] (0,0) -- (0,1);
\draw[color = red, ultra thick] (0,0) -- (1,0);
\draw[color = red, ultra thick] (0,1) -- (1,1);
\draw[color = blue, dashed, ultra thick] (1,0) -- (1,1);

\draw[color = blue, dashed, ultra thick] (2,0) -- (2,1);
\draw[color = red, ultra thick] (2,0) -- (3,0);
\draw[color = red, ultra thick] (2,1) -- (3,1);
\draw[color = blue, dashed, ultra thick] (3,0) -- (3,1);
\draw[color = green, densely dotted, ultra thick] (2,0) -- (3,1);
\draw[color = green, densely dotted, ultra thick] (2,1) -- (3,0);

\draw [-{Implies},double, line width = 1pt] (1.2,0.5) -- (1.8,0.5);

 \fill (0,0) node[left]{$B$} circle (1.5pt);
 \fill (0,1) node[left]{$A$} circle (1.5pt);
 \fill (1,0) node[right]{$C$} circle (1.5pt);
 \fill (1,1) node[right]{$D$} circle (1.5pt);
 
  \fill (2,0) node[left]{$B$} circle (1.5pt);
 \fill (2,1) node[left]{$A$} circle (1.5pt);
 \fill (3,0) node[right]{$C$} circle (1.5pt);
 \fill (3,1) node[right]{$D$} circle (1.5pt);
 
\end{tikzpicture}
\caption{A graphical illustration of Lemma \ref{lem:c4}.}
\label{fig:C4}
\end{figure}
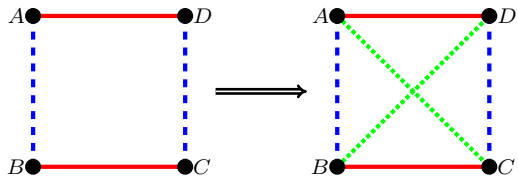

\begin{lem}[\cite{Cherowitzo}]
\label{lem:c4}
Let $\mathcal{F}$ be the 1-factorization obtained from the focus set of a hyperfocused arc in PG$(2,2^h)$. If there are two 1-factors in $\mathcal{F}$ with a $C_4$ in their union, then there must be a unique third 1-factor that completes the $C_4$ to a $K_4$.
\end{lem}

\begin{proof}
Let $\mathcal{H}$ be a hyperfocused $k$-arc in PG($2,2^h$) that has focus line $\ell$. Let $P$, $Q$ be two points on $\ell$ associated with 1-factors in $\mathcal{F}$ with a $C_4$ in their union. We may assume that $P = AB \cap CD$ and $Q = AD \cap BC$ where $A,B,C,D \in \mathcal{H}$. The diagonal line of the quadrangle ABCD contains both $P$ and $Q$ so must be $\ell$. Hence the point $R = AC \cap BD$ is associated with the desired 1-factor. \qed
\end{proof}

We add the following necessary condition for a 1-factorization to embed in PG($2,2^h$). We make significant use of Desargues' theorem.

\begin{thm}[Desargues' Theorem]
Two triangle are in perspective from a point if and only if they are in perspective from a line.
\end{thm}

\begin{lem}
\label{lem:2k4}
Let $\mathcal{F}$ be the 1-factorization obtained from the focus set of a hyperfocused arc in PG$(2,2^h)$. If there are two disjoint copies of $K_4 -e$ with the same coloring (see Figure \ref{fig:K4}), then the remaining edge of each $K_4$ both have the same color as well.
\end{lem}

\begin{proof}

Let $\mathcal{H}$ be a hyperfocused arc with focus line $\ell$, and let $\mathcal{F}$ be the 1-factorization obtained from the focus set of $\mathcal{H}$. Additionally let $\mathcal{F}_1,\dots,\mathcal{F}_5$ be 1-factors in $\mathcal{F}$. Assume that $A,B,C,D,E,F,G,H \in \mathcal{H}$, and that $(AB), (EF) \in \mathcal{F}_1$, $(AC), (EG) \in \mathcal{F}_2$, $(BC), (FG) \in \mathcal{F}_3$, $(BD), (FH) \in \mathcal{F}_4$, and $(CD), (GH) \in \mathcal{F}_5$, where $(XY)$ denotes the edge between the vertices corresponding to the points on $\mathcal{H}$. See Figure \ref{fig:K4}. 

Observe that the triangles $\Delta ABC$ and $\Delta EFG$ are in perspective from the line $\ell$ and so must be in perspective from a point. Call that point $V$. Now, triangles $\Delta BCD$ and $\Delta FGH$ are also in perspective from $\ell$ and so must also be in perspective from $V$. Observe that triangles $\Delta ACD$ and $\Delta EGH$ are also in perspective from $V$ and so must be in perspective from a line. Since $AC$ meets $EG$ on $\ell$ and $CD$ meets $GH$ on $\ell$ we must have that $AD$ meets $EH$ on $\ell$. Thus $(AD)$ and $(EH)$ are in the same 1-factor (hence have the same color), which completes both $K_4-e$ to a $K_4$.\qed
\end{proof}

\begin{figure}
\centering
\begin{tikzpicture}[scale = 2]

\draw[color = blue, dashed, ultra thick] (0,0) -- (0,1);
\draw[color = red, ultra thick] (0,0) -- (1,0);
\draw[color = green, densely dotted, ultra thick] (0,0) -- (1,1);
\draw[color = black, ultra thin] (0,.975) -- (.975,0);
\draw[color = black, ultra thin] (.025,1) -- (1,.025);
\draw[color = orange, dashdotted, ultra thick] (1,0) -- (1,1);

\draw[color = blue, dashed, ultra thick] (2,0) -- (2,1);
\draw[color = red, ultra thick] (2,0) -- (3,0);
\draw[color = green, densely dotted, ultra thick] (2,0) -- (3,1);
\draw[color = black, ultra thin] (2,.975) -- (2.975,0);
\draw[color = black, ultra thin] (2.025,1) -- (3,.025);
\draw[color = orange, dashdotted, ultra thick] (3,0) -- (3,1);

 \fill (0,0) node[left]{$B$} circle (1.5pt);
 \fill (0,1) node[left]{$A$} circle (1.5pt);
 \fill (1,0) node[right]{$C$} circle (1.5pt);
 \fill (1,1) node[right]{$D$} circle (1.5pt);
 
  \fill (2,0) node[left]{$F$} circle (1.5pt);
 \fill (2,1) node[left]{$E$} circle (1.5pt);
 \fill (3,0) node[right]{$G$} circle (1.5pt);
 \fill (3,1) node[right]{$H$} circle (1.5pt);

\end{tikzpicture}
\caption{ Two corresponding $K_4 - e$.}
\label{fig:K4}
\end{figure}
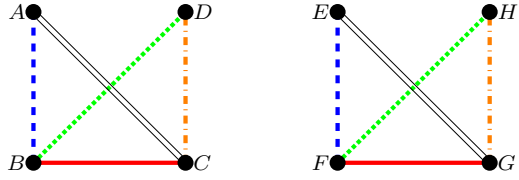

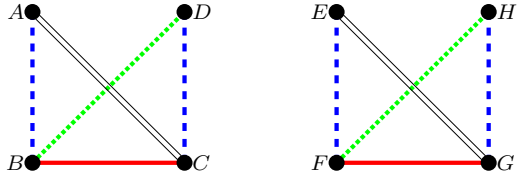
\begin{figure}
\centering
\begin{tikzpicture}[scale = 2]

\draw[color = blue, dashed, ultra thick] (0,0) -- (0,1);
\draw[color = red, ultra thick] (0,0) -- (1,0);
\draw[color = green, densely dotted, ultra thick] (0,0) -- (1,1);
\draw[color = black, ultra thin] (0,.975) -- (.975,0);
\draw[color = black, ultra thin] (.025,1) -- (1,.025);
\draw[color = blue, dashed, ultra thick] (1,0) -- (1,1);

\draw[color = blue, dashed, ultra thick] (2,0) -- (2,1);
\draw[color = red, ultra thick] (2,0) -- (3,0);
\draw[color = green, densely dotted, ultra thick] (2,0) -- (3,1);
\draw[color = black, ultra thin] (2,.975) -- (2.975,0);
\draw[color = black, ultra thin] (2.025,1) -- (3,.025);
\draw[color = blue, dashed, ultra thick] (3,0) -- (3,1);

 \fill (0,0) node[left]{$B$} circle (1.5pt);
 \fill (0,1) node[left]{$A$} circle (1.5pt);
 \fill (1,0) node[right]{$C$} circle (1.5pt);
 \fill (1,1) node[right]{$D$} circle (1.5pt);
 
  \fill (2,0) node[left]{$F$} circle (1.5pt);
 \fill (2,1) node[left]{$E$} circle (1.5pt);
 \fill (3,0) node[right]{$G$} circle (1.5pt);
 \fill (3,1) node[right]{$H$} circle (1.5pt);

\end{tikzpicture}
\caption{Another example of corresponding $K_4 - e$.}
\label{fig:K4alt}
\end{figure}

The following theorems from Cherowitzo and Holder, which classify all hyperfocused arcs contained in a hyperconic and contain the nucleus, as well as Pascal's theorem, are the final pieces of information needed for us to complete our classification.

\begin{thm}[Pascal's Theorem]
The six points of a hexagon lie on a conic if and only if the pairs of opposite sides meet in three colinear points.
\end{thm}

\begin{thm}[\cite{Cherowitzo}]
\label{thm:secant}
A set of points $K$ with $3 < | K |$ on a hyperconic in PG($2, q), q = 2^h$ which includes the nucleus $N$ of the conic is hyperfocused on a secant line to that conic which does not meet $K$ if and only if $K \setminus \{ N \}$ is projectively equivalent to a set of points determined by a subgroup of (GF($q)^*$,$\cdot$).
\end{thm}

\begin{thm}[\cite{Cherowitzo}]
\label{thm:external}
A set of points $K$ with $3 < | K |$ on a hyperconic in PG($2, q), q = 2^h$ which includes the nucleus $N$ of the conic is hyperfocused on an exterior line to that conic if and only if $K \setminus \{ N \}$ is projectively equivalent to a set of points determined by a subgroup of $\mathbb{Z}_{q+1}$.
\end{thm}

\section{Classifying Hyperfocused 12-arcs}

In order to classify hyperfocused 12-arcs we examined the list of 1-factorizations of $K_{12}$. This list was graciously provided in usable form by Kaski and {\"O}sterg{\aa}rd from their paper \cite{Kaski}. We studied each 1-factorization and checked, via computer, the necessary conditions from Lemma \ref{lem:c4} and Lemma \ref{lem:2k4}. The code can be found at \cite{Hartke}. 

When performing the computer search we break the computation down into two parts.
The first is a filter of the $526,915,620$ 1-factorizations of $K_{12}$ to determine which satisfy Lemma~\ref{lem:c4}. This filtering produced a list of only $253$ 1-factorizations.  We then checked the list of 253 against the conditions of Lemma~\ref{lem:2k4}, which yielded a list of $2$ remaining 1-factorizations.

To filter using Lemma~\ref{lem:c4}, we used the following process.
In each 1-factorization, we examine each 1-factor $\mathcal{F}$ and each pair $P$ of edges from $\mathcal{F}$.
The pair $P$ contains four vertices that we use to check Lemma~\ref{lem:c4}.
There are two pairs of disjoint edges between these vertices that are not in $\mathcal{F}$.  
Lemma~\ref{lem:c4} states that if one of the pairs of disjoint edges is contained in another 1-factor, then the other pair must also be contained in a third 1-factor.
So we check if exactly one of the pairs is contained in a 1-factor, and if so, exclude the 1-factorization from our list.  Otherwise the 1-factorization makes it through our filter.
As noted above, there are $253$ such 1-factorizations that made it through this step.
As an optimization, in a given $1$-factorization we check every $1$-factor as described above except for one.  Since two 1-factors are needed for a violation of Lemma~\ref{lem:c4}, the method described above will detect the violation starting from either 1-factor.

The next part of the computation is to check the surviving 1-factorizations against Lemma~\ref{lem:2k4}. For each 1-factorization, we examine each set of four unordered vertices from $K_{12}$ and then another set of four ordered vertices disjoint from the original four.
These vertices make up the vertices of our two disjoint $K_4$s. 
The choosing of unordered and ordered sets of vertices accounts for the different ways the edges may correspond.
Next, we counted how many pairs of corresponding edges between the two $K_4$s have the same color, and if there are exactly five pairs of corresponding edges with the same color, Lemma~\ref{lem:2k4} is violated and we exclude that 1-factorization from our list. 
At the end of this process we were left with only 2 1-factorizations.

We now consider both of these 1-factorizations which are given below. We identify the vertices of the complete graph $K_{12}$ with the integers $0,1,\dots,11$ and the edges as order pairs $(x,y)$ with $x,y \in \{0,1,\dots,11\}$. Further, we identify the 1-factors with letters $A,\dots,K$. The first 1-factorization we want to examine is:
\vspace{.1in}
\begin{center}
\noindent
A: [(0,1), (2,3), (4,5), (6,7), (8,9), (10,11)]\\
B: [(0,2), (1,3), (4,6), (5,7), (8,10), (9,11)]\\
C: [(0,3), (1,2), (4,7), (5,6), (8,11), (9,10)]\\
D: [(0,4), (1,5), (2,8), (3,9), (6,11), (7,10)]\\
E: [(0,5), (1,4), (2,10), (3,11), (6,8), (7,9)]\\
F: [(0,9), (1,8), (2,4), (3,5), (6,10), (7,11)]\\
G: [(0,11), (1,10), (2,5), (3,4), (6,9), (7,8)]\\
H: [(0,6), (1,7), (2,11), (3,10), (4,8), (5,9)]\\
I: [(0,7), (1,6), (2,9), (3,8), (4,11), (5,10)]\\
J: [(0,10), (1,11), (2,6), (3,7), (4,9), (5,8)]\\
K: [(0,8), (1,9), (2,7),(3,6), (4,10), (5,11)]\\
\end{center}
\vspace{.1in}
Assume that this 1-factorization embeds and without loss of generality assume that we have coordinates $A = (1,0,0)$, $B = (0,1,0)$, ${\bf 0 } = (0,0,1)$, and ${\bf 3} = (1,1,1)$. Thus the focus line is $[0,0,1]^T$ and since $(03)$ meets the focus line at $C$ we must have $C = (1,1,0)$. Completing the Fano plane $0123ABC$ yields ${\bf 1} = (1,0,1)$, ${\bf 2} = (0,1,1)$. The point ${\bf 4}$ cannot be on the line $[0,0,1]^T$ so assume it has coordinates $(x,y,1)$. Observe that $(45)$ passes through $A$ and $(46)$ passes through $B$ so we must have coordinates ${\bf 5} = (s,y,1)$ and ${\bf 6} = (x,t,1)$. Since $4567ABC$ forms a Fano plane we must also have ${\bf 7 } = (s,t,1)$ with $x + s = y + t$. Since $D = (04) \cap (1,5)$ we must have that $D = (x,y,s+x+1)$ and using the fact that the focus line is $[0,0,1]^T$ we get $s = x+1$ and consequently $t = y +1$. Further we have $E = (x+1,y,0)$, $F = (x,y+1,0)$, and $G = (x+1,y+1,0)$. Now, $(06)$ meets the focus line at $H$, so $H = (x,y+1,0)$, forcing $H = F$, a contradiction. Thus this 1-factorization is not embeddable. We now consider the only remaining 1-factorization.

\vspace{.1in}
\begin{center}
\noindent
A: [(0,1), (2,3), (4,5), (6,7), (8,9), (10,11)]\\
B: [(0,2), (1,4), (3,6), (5,8), (7,10), (9,11)]\\
C: [(0,11), (1,2), (3,4), (5,6), (7,8), (9,10)]\\
D: [(0,3), (1,5), (2,7), (4,9), (6,11), (8,10)]\\
E: [(0,10), (1,3), (2,5), (4,7), (6,9), (8,11)]\\
F: [(0,4), (1,8), (2,6), (3,10), (5,11), (7,9)]\\
G: [(0,9), (1,6), (2,4), (3,8), (5,10), (7,11)]\\
H: [(0,5), (1,9), (2,11), (3,7), (4,10), (6,8)]\\
I: [(0,8), (1,7), (2,9), (3,5), (4,11), (6,10)]\\
J: [(0,6), (1,11), (2,10), (3,9), (4,8), (5,7)]\\
K: [(0,7), (1,10), (2,8), (3,11), (4,6), (5,9)]\\
\end{center}
\vspace{.1in}
The hexagons 236541, 231547, 23754[10], 387421, 139748, 13[11]746 show that the points 1-11 lie on a conic $\mathcal{C}$ by the converse of Pascal's Theorem and the fact that five points determine a conic. The hexagon 015327 shows that 0 does not lie on $\mathcal{C}$. The triangles 123, 0[10][11] are in perspective from the focus line so the lines (2[11]), (01), (3[10]) meet at a point $X$. The quadrangle 23[10][11] has diagonal line (AX) showing (01) is a tangent line of the conic. The triangles 014, 326 are in perspective from the focus line, so the lines (13), (02), (46) meet at a point $Y$. The quadrangle 1346 has diagonal line (BY) showing (02) is a tangent line of the conic. Hence, $0$ is the nucleus of $\mathcal{C}$, as it is the intersection of two tangent lines.

Since 1-11 lie on a conic we will assume that the equation of the conic is $x^2 = yz$. The nucleus of this conic is $(1,0,0)$ so let {\bf 0} = $(1,0,0)$ and WLOG let {\bf 1} =$ (0,1,0)$, {\bf 2} = $(0,0,1)$, {\bf 3} = $(1,1,1)$. Now let $a \in GF(q) - \{ 0,1 \}$ and let {\bf 4} = $(a,a^2,1)$. As $A$ is the intersection of $(01)$ and $(23)$ it has coordinates $(1,1,0)$, and $B$ is the intersection of $(02)$ and $(14)$ so it has coordinates $(a,0,1)$. Thus our focus line is $[1,1,a]^T$.

Using the equation of the conic, the line $[1,1,a]^T$ and the known coordinates we can determine the coordinates of the remaining points on the conic, and the points in the focus set shown in Table \ref{fig:methods}. To assist the reader we note the methods by which the coordinates of the points above were determined.

\begin{table}
\centering
\resizebox{\columnwidth}{!}{%
\begin{tabular}{|l|l|l||l|l|l|} \hline
Point & Coordinates & Determined by & Point & Coordinates & Determined by \\ \hline
{\bf 0 } & $(1,0,0)$ & Nucleus & $A$ & $(1,1,0)$ & $(02) \cap (13)$ \\ \hline
{\bf 1 } & $(0,1,0)$ & Assumed & $B$ & $(a,0,1)$ & $(02) \cap (14) $\\ \hline
{\bf 2 } & $(0,0,1)$ & Assumed & $C$ & $(0,a,1)$ & $(12) \cap [1,1,a]^T$ \\ \hline
{\bf 3 } & $(1,1,1)$ & Assumed & $D$ & $(a+1,1,1)$ & $(03) \cap [1,1,a]^T$ \\ \hline
{\bf 4 } & $(a,a^2,1)$ & Assumed & $E$ & $(1, a+1, 1)$ & $(13) \cap [1,1,a]^T$ \\ \hline
{\bf 5 } & $(a+1, a^2+1, 1)$ & $(D1) \cap \mathcal{C}$ & $F$ & $(a^2 + a, a^2, 1)$ & $(04) \cap [1,1,a]^T$\\ \hline
{\bf 6 } & $(a^2 + a, a^2, a^2 + 1)$ & $(B3) \cap (F2)$ & $G$ & $(a, a^2, a+1)$ & $(24) \cap [1,1,a]^T$\\ \hline
{\bf 7 } & $(a+1, 1, a^2 + 1)$ & $(D2) \cap (E4)$ & $H$ & $(a^2 + a + 1, a^2 + 1, 1)$ & $(05) \cap [1,1,a]^T$ \\ \hline
{\bf 8 } & $(a^2 + a, a^4 + a^2, 1)$ & $(F1) \cap \mathcal{C}$ & $I$ & $(1, a^2 + a + 1, a+1)$ & $(35) \cap [1,1,a]^T$ \\ \hline
{\bf 9 } & $(a^2 + a + 1, a^4 + a^2 + 1, 1)$ & $(H1) \cap \mathcal{C}$ & $J$ & $(a^5 + 1, a^5 + a^3, a^5+1)$ & $(06) \cap (39)$ \\ \hline
{\bf 10}  & $(a^3+a+1,a^3+a^2+a+1,a^2+1)$ & $(E0) \cap (F3)$ & $K$ & $(1, a^2 + a, a^4 + a^3 + a^2 + a)$ & $(07) \cap (28)$ \\ \hline
{\bf 11}  & $(a^3 + a^2 + a, a^3 + a, a^2 + 1)$ & $(C0) \cap (H2)$ & & & \\ \hline 
\end{tabular}%
}
\caption{Coordinates of the hyperfocused arc and focus set.}
\label{fig:methods}
\end{table}

An alternative way to compute the coordinates of point {\bf 8} is to consider the intersection of the lines $(F1)$ and $(I0)$. When performing this computation we get that the coordinates of {\bf 8} must be $(a^3 + a, a^2 + a + 1, a+1)$. Multiplying the coordinates of {\bf 8} from Table \ref{fig:methods} by $(a+1)$ gives us that the coordinates of {\bf 8} must also be $(a^3 + a, a^5 + a^4 + a^3 + a^2, a+1)$ and thus $a^5 + a^4 + a^3 + a^2 = a^2 + a + 1$. Rearranging the terms yields $a^5 =  a^4 + a^3 + a + 1$. Therefore $a$ is a root of the polynomial $x^5 +  x^4 + x^3 + x + 1$, which is irreducible over $GF(2)$. This implies $GF(2)[a] = GF(2^5)$, hence our coordinatizing field $GF(q)$ must contain $GF(2^5)$, from which we conclude $q = 2^{5k}$.  

The next lemma will help classify this hyperfocused arc. Let $T$ denote the absolute trace function.

\begin{lem}
\label{lem:Trace}
If $a \in GF(2^{5k})$ satisfying  $a^5 = a^4 + a^3 + a + 1$ then $T(a) = T(1)$.
\end{lem}

\begin{proof}
Since finite fields are separable extensions of their ground field, we use the formula $$T(a) = [GF(q):GF(2)] \sum\limits_{ \alpha \in R } \alpha,$$ where $R$ is the set of roots of the minimal polynomial of $a$. Since $x^5 +  x^4 + x^3 + x + 1$ is irreducible, it must be the minimal polynomial of $a$. The coefficient of $x^4$ is 1 in this polynomial, thus we have $\sum\limits_{ \alpha \in R } \alpha  = 1$. 
Therefore, $T(a) = [GF(q):GF(2)] = T(1)$, as desired. \qed
\end{proof}

It is well known that the polynomial $t^2 + t + a$ is irreducible if and only if $T(a) = 1$. Lemma \ref{lem:Trace} shows that the polynomial $t^2 + t + a$ is irreducible over $GF(2^{5k})$ if and only if $k$ is odd. Hence the line $[1,1,a]^T$ is external to $\mathcal{C}$ when $k$ is odd, and secant to $\mathcal{C}$ when $k$ is even. Thus our classification follows from Theorems \ref{thm:secant} and \ref{thm:external}.\

\begin{thm}
A 12-arc $\mathcal{K}$ is a hyperfocused arc in PG($2,q$) if and only if $q = 2^{5k}$ and $\mathcal{K}$ is a subset of a hyperconic including the nucleus. Specifically, the focus line is secant to the conic when $k$ is even so the non-nucleus points are projectively equivalent to a set of points determined by a subgroup of (GF($q)^*,\cdot$); and the focus line is exterior when $k$ is odd, so the non-nucleus points are projectively equivalent to a set of points determined by a subgroup of $\mathbb{Z}_{q+1}$.
\end{thm}

\section*{Acknowledgements}

This research started at the Rocky Mountain--Great Plains Graduate Research Workshop in Combinatorics that was held at the University of Colorado Denver and the University of Denver in summer 2014, and we thank them for their hospitality.
We thank Petteri Kaski and Patric {\"O}sterg{\aa}rd for providing us with the data that allowed us to complete this classification.
We also thank William Cherowitzo for his many helpful remarks.


\begin{thebibliography}{99}

\bibitem{Korchmaros}
A. Bichara and G. Korchm{\'a}ros 
\emph{Note on (q+2)-sets in a galois plane of order q.}
Annals of Discrete Math. 14 (1982), 117--122.

\bibitem{Cherowitzo} 
Cherowitzo,W.E., Holder,L.D. 
\emph{Hyperfocused arcs.} 
Simon Stevin 12 (2005), no. 5, 685--696.

\bibitem{Dinitz}
Dinitz, Jeffrey H., Garnick, David K., McKay, Brendan D.
\emph{There are 526,915,620 nonisomorphic one-factorizations of $K_{12}$}. 
J. Combin. Des. 2 (1994), no. 4, 273--285.

\bibitem{Drake}
Drake, D., Keating, K. 
\emph{Ovals and hyperovals in Desarguesian nets.} 
Des. Codes Cryptogr. 31 (2004), no. 3, 195--212. 

\bibitem{Faina}
Faina, Giorgio, Parrettini, Cristiano, Pasticci, Fabio.
\emph{Embedding 1-factorizations of $K_n$ in PG(2,32).} 
Graphs Combin. 29 (2013), no. 4, 883--892.

\bibitem{Giulietti}
Giulietti, M., Montanucci, E.
\emph{On hyperfocused arcs in PG(2,q).} 
Discrete Math. 306 (2006), no. 24, 3307--3314.

\bibitem{Hartke}
DeOrsey, P., Hartke, S., Williford, J., GitHub Repository, \url{https://github.com/hartkes/hyperfocused12arcs}.

\bibitem{Holder}
Holder, L.D.
\emph{The construction of Geometric Threshold Schemes with Projective Geometry.}
Master's Thesis, University of Colorado at Denver (1997).

\bibitem{Kaski}
Kaski, Petteri, {\"O}sterg{\aa}rd, Patric R. J.
\emph{There are 1,132,835,421,602,062,347 nonisomorphic one-factorizations of $K_{14}$.} 
J. Combin. Des. 17 (2009), no. 2, 147--159.

\bibitem{Segre}
Segre, B.
\emph{Sulle ovali nei piani lineari finite.} 
Rend. Accad. Naz. Lincei 17 (1954), 141--142.

\bibitem{Simmons}
Simmons, G.
\emph{Sharply focused sets of lines on a conic in PG(2,q).}
Congr. Numer. 73 (1990), 181--204. 

\bibitem{West}
West, D.B.
\emph{Introduction to Graph Theory.}
Prentice Hall, 2nd edition, 2001.

\end{thebibliography}
\end{document}